\pdfoutput=1

\RequirePackage [l2tabu, orthodox]{nag}
\documentclass[12pt, reqno]{amsart}
\usepackage[utf8]{inputenx}
\usepackage{amssymb, fullpage, amsmath, mathtools, graphicx, graphics, xcolor, soul}
\usepackage[a4paper,margin=0.8in]{geometry}
\usepackage[colorlinks]{hyperref}
\usepackage[alphabetic]{amsrefs}
\hypersetup{citecolor=blue}
\usepackage{setspace}

\numberwithin{equation}{section}
\renewcommand{\S}{\mathbb S}
\renewcommand{\d}{\mathcal D}
\newcommand{\Z}{\mathbb Z}
\newcommand{\R}{\mathbb R}

\newcommand{\N}{\mathbb N}
\renewcommand{\H}{\mathbb H}
\newcommand{\T}{\mathbb T}
\newcommand{\eps}{\varepsilon}
\newcommand{\e}{\mathbf e}

\newcommand{\qede}{\hfill $\diamond$}

\theoremstyle{plain}
\newtheorem{thm}{Theorem}[section]
\newtheorem{cor}{Corollary}[section]

\newtheorem{prop}{Proposition}[section]

\theoremstyle{definition}
\newtheorem{defin}{Definition}[section]

\newtheorem{remark}{Remark}[section]
\newtheorem{exa}{Example}[section]

\begin{document}

\title{A Perron\textendash Frobenius type result for integer maps \\ and applications}

\dedicatory{Dedicated to the memory of Jonathan M.\ Borwein (1951-2016)}

\begin{abstract}
  It is shown that for certain maps, including concave maps, on the $d$-dimensional lattice of positive
  integer points, `approximate' eigenvectors can be found.  Applications in epidemiology as well as
  distributed resource allocation are discussed as examples.
\end{abstract}
\author{Ohad Giladi}

\author{Bj\"orn S.\ R\"uffer}

\address{School of Mathematical and Physical Sciences, University of Newcastle, Callaghan, NSW 2308, Australia}

\email{ohad.giladi@newcastle.edu.au, bjorn.ruffer@newcastle.edu.au}

% \date{\today}

\subjclass[2010]{37J25, 92D30, 93D20}

\keywords{Perron\textendash Frobenius theory, integer maps, concave maps, Hilbert metric}

\maketitle

\section{Introduction}

The classical linear Perron\textendash Frobenius theorem goes back to the work of Oskar Perron, who studied the eigenvalue problem
$Ax=\lambda x$ for positive matrices, and later Georg Frobenius, who extended the result to non-negative irreducible
matrices. The theorem asserts the existence of a positive eigenvalue equal to the spectral radius and a corresponding
positive, respectively, non-negative eigenvector. This theorem has found various applications in economics, the study of
Markov chains, differential equations, and more. A detailed discussion of this and related results can be found, e.g.,
in the books~\cites{BP94, Gan59}.

It is also possible to study eigenvectors in a nonlinear setting. Nonlinear Perron\textendash Frobenius results appeared
already in~\cite{SS53}. Later, a new approach to the eigenvalue problem was introduced in the work of
Birkhoff~\cite{Bir57} and Samelson~\cite{Sam57}. This new approach enabled the study of eigenvalues and
eigenvectors for a large class of nonlinear positive maps. There is now a rich literature on Perron\textendash Frobenius
results for nonlinear positive maps. One area in which Perron\textendash Frobenius theory has been found useful is the
area of economic theory, for example in questions related to price stability, cf.,
e.g.,~\cite{Koh82}*{Sec.~2}. Indeed, many of the Perron\textendash Frobenius results, such as~\cites{Koh82, Kra86, MF74,
  SS53} appeared in economics journals, cf., \cites{Mor64, Nik68}.

Convex and concave maps appear often in economics, cf., e.g.,~\cite{Nik68}, so quite a few Perron\textendash Frobenius
type results have been studied for this case, e.g., by~\cite{Kra86}. The notion of concavity can also be
studied in a discrete setting, cf.~\cite{BG18}.  Other Perron\textendash Frobenius type results can be found
in~\cites{Ch14, KP82, Kra01, Nus88} to mention just a few. Finally, the book~\cite{LN12} gives a good introduction to
the nonlinear Perron\textendash Frobenius theory.

The approach introduced in~\cites{Bir57, Sam57} can be described in the following way. Suppose that $A$ is a
positive map in a $d$-dimensional space, that is $A:\R_+^d \to \R_+^d$, where here and in what follows
$\R_+ = [0,\infty)$ denotes all non-negative real numbers. The key idea is to consider the normalized map
$Bx = Ax / \|Ax\|$, where $\|\cdot\|$ is some norm on $\R^d$, and then to show that with respect to a given
metric (typically the Hilbert projective metric, see Section~\ref{sec main thm} for the precise definition),
the map $B$ is well behaved and leaves invariant some compact subset of $\R_+^d$. Then, using a fixed
point result such as the Banach contraction principle or the Brouwer fixed point theorem, it follows that there
exists a vector $x\in \R_+^d$ such that $Bx = x$. This fixed point is the desired eigenvector, as we have
$Ax = \|Ax\|\, x$.

In this paper we consider maps $A:\Z_+^d \to \Z_+^d$, where here and in what follows $\Z_+ = \{0,1,\dots\}$
denotes all non-negative integers, while $\N = \{1,2, \dots\}$.  Where previous results focus on maps defined
on $\R_+^d$ (and in some cases on infinite dimensional Banach spaces), here we show that with reasonable
adaptation of the existing tools it is possible to study maps in a discrete setting.  As many of the classical
applications of Perron\textendash Frobenius theory are in fact continuous approximations of discrete models, the use of a
discrete Perron\textendash Frobenius theory gives a different, more direct, approach to dealing with such problems.

To this end, in order to use fixed point theorems, we show that under certain conditions the map $A$ can be extended to
a well behaved map on a compact, convex set in $\R_+^d$. Since the fixed point of the extended map $B$ may be a
non-integer point, we only have an `approximate' eigenvector, which is suitably characterized by inequalities.  A notion
of concavity, originally introduced for groups in~\cite{BG18}, is used to study discrete, concave maps.

This paper is organized as follows. The main result is given in Section~\ref{sec main thm}.  Theorem~\ref{main
  thm} extends the classic Perron\textendash Frobenius theorem to discrete maps on the $d$-dimensional positive integer
lattice. Corollary~\ref{cor sequence} shows that under the assumption that the norm of $Ax$ is well behaved, we
may obtain a sequence in $\Z_+^d$ with controlled growth or decay.  In Section~\ref{sec concave} we show that
the main result can be applied to concave maps on $\Z_+^d$.  In Section~\ref{sec app}, it is shown how the
results of Section~\ref{sec main thm} and Section~\ref{sec concave} can be applied to models from biology and
engineering. In particular, we study a discrete variant of the \emph{Susceptible-Infected-Susceptible} (SIS) model, as
well as two models from communications: the \emph{Additive Increase Multiplicative Decrease} (AIMD) model and
an interference constraints model for wireless communication systems.

\subsection*{Acknowledgements} This paper was written while both
authors were members of the priority research centre for
Computer-Assisted Research Mathematics and its Applications (CARMA) at
the University of Newcastle, Australia (UON). CARMA was founded in
2009 by Jonathan M.\ Borwein, who also served as its director. Jon was
a prolific researcher and a devoted friend. This paper is dedicated to
his memory with admiration.

\section{Approximate eigenvectors for integer maps}\label{sec main thm}

Before we can state and prove the main result, we recall some basic notations that will be used throughout this
paper. Let $\R_+^d$ denote the non-negative orthant in $\R^d$, which is a cone.  Given two vectors $x=(x_1,\dots,x_d)$
and $y = (x_1,\dots,x_d)$ in $\R_+^d$, let $\le$ denote the standard (component-wise) partial order induced by this
cone, that is,
\begin{align*}
  x & \le y &&\iff& y-x&\in\R_{+}^{d} &&\iff &\big[ x_i \le y_i ~ \forall i \in \{1,\dots,d\}\big],\\
\intertext{and also denote}
  x &\ll y &&\iff &y-x&\in(0,\infty)^{d} &&\iff &\big[ x_i < y_i~   \forall i \in \{1,\dots,d\}\big].
\end{align*}
Note that the maximum and minimum with respect to the cone partial order coincide with component-wise maximum
and minimum.

Given $x,y \in \R^d_{+}\setminus \{0\}$, define
\begin{align}
  \label{def lambda}
  \lambda(x,y) = \sup\big\{\gamma \in \R_+ \,\big\vert\, \gamma x \le y\big\}.
\end{align}
Define the Hilbert metric on $\R_+^d\setminus\{0\}$ by
\begin{align}
  \label{def Hilbert}
  d_{\H}(x,y) = \begin{dcases} -\log\big(\lambda(x,y)\lambda(y,x)\big) & \text{for }\lambda(x,y)\lambda(y,x) >0, \\ \infty &  \text{otherwise.}\end{dcases}
\end{align}
Recall also the $\ell_p$ norm on $\R^d$, denoted by $\|\cdot\|_p$ and given by
\begin{align*}
  \|x\|_p =
  \begin{dcases}
    \left(\sum_{j=1}^d|x_j|^p\right)^{1/p} &\text{for } p \in [1,\infty), \\
    \max_{1 \le j \le d}|x_j| &\text{for }  p = \infty.
  \end{dcases}
\end{align*}
Let $\e$ denote the vector $(1,1,\dots, 1) \in \R_+^d$, and let $\e_1,\dots,\e_d$ denote the standard basis vectors in $\R_+^d$, i.e., $\e_1 = (1,0,\dots,0)$, $\e_2 = (0,1,0,\dots,0)$, etc.

In~\cite{Kra86}, the following relation between the Hilbert metric $d_{\H}$ and the $\ell_\infty$ norm
was proven.

\begin{prop}
  \label{prop equiv kra}
  Assume that $x,y\in \R_+^d$ are such that $\|x\|_{1}=\|y\|_{1}>0$ and 
  $\frac{1}{\beta}\max\{x,y\}\leq \e \leq \frac{1}{\alpha}\min\{x,y\}$
  for some $\alpha, \beta \in (0,\infty)$. Then
  \begin{align*}
    \alpha\left(1-e^{-\frac{d_{\H}(x,y)}{2}}\right) \le \|x-y\|_{\infty} \le
    \beta\left(1-e^{-d_{\H}(x,y)}\right).
  \end{align*}
\end{prop}

In~\cite{Kra86}, it was additionally assumed that $\|x\|_1 = \|y\|_1 = 1$, but using exactly the same proof,
the result holds under the weaker assumption that $\|x\|_1 = \|y\|_1$ are positive. Proposition~\ref{prop equiv
  kra} immediately implies the following result.

\begin{prop}\label{prop equiv} Assume that $x,y \in \R_+^d$ are such that $\|x\|_1 = \|y\|_1>0$ and
  $\frac{1}{\beta}\max\{x,y\}\leq \e \leq \frac{1}{\alpha}\min\{x,y\}$
  for some $\alpha, \beta \in (0,\infty)$. Then
  \begin{align}\label{equiv dist infty} \frac {\alpha^2}{2\beta}d_{\H}(x,y) \le \|x-y\|_{\infty} \le \beta
    d_{\H}(x,y),
  \end{align}
  and
  \begin{align}\label{equiv dist 2} \frac {\alpha^2}{2\beta}d_{\H}(x,y) \le \|x-y\|_{2} \le \sqrt {d}\beta
    d_{\H}(x,y).
  \end{align}
\end{prop}

\begin{proof}
  To prove~\eqref{equiv dist infty}, note that since 
  $\frac{1}{\beta}\max\{x,y\}\leq \e \leq \frac{1}{\alpha}\min\{x,y\}$,
  we
  have $x \ge \frac{\alpha}{\beta}y$, $y \ge \frac{\alpha}{\beta}x$, and so $d_{\H}(x,y) \le
  2\log\left(\frac{\beta}{\alpha}\right)$. Since $1-e^{-t} \le t$ for all $t \in \R_+$ and also $1-e^{-t/2} \ge
  \frac{\alpha}{2\beta} t$ whenever $t\ \le 2\log\left(\frac{\beta}{\alpha}\right)$, the result follows from Proposition~\ref{prop equiv kra}.
  Next,~\eqref{equiv dist 2} follows from the fact the for every $x\in \R^d$, $\|x\|_\infty \le \|x\|_2 \le \sqrt
  d \|x\|_\infty$, and this completes the proof.
\end{proof}

Another tool which is needed in the proof of Theorem~\ref{main thm} is an extension result for Lipschitz maps. Given
two metric spaces $(X,d_X)$ and $(Y,d_Y)$, a map $f:X\to Y$ is said to be $L$-Lipschitz if there exists $L\in [0,\infty)$ such that for every
$x,y \in X$,
\begin{align*}
  d_Y\big(f(x),f(y)\big) \le Ld_X(x,y).
\end{align*}
Given a subset $Z\subseteq X$ and an $L$-Lipschitz map $f:Z\to Y$, a well-studied question is whether $f$ can be
extended onto all of $X$, while preserving the Lipschitz property. In the case where $(X,d_X)$ and $(Y,d_Y)$
are Hilbert spaces, the following theorem is a well-known result due to Kirszbraun, cf.,
e.g.,~\cite{GK90}*{Thm.~12.4}.

\begin{thm}[Kirszbraun]\label{thm kirsz} Assume that $D_1, D_2\subseteq \R^d$ and $f:D_1\to D_2$ is such that
  \begin{align*} \|f(x)-f(y)\|_2 \le L\|x-y\|_2.
  \end{align*}
  Then there exists $\tilde f :\R^d\to \overline{\mathrm{conv}(D_2)}$ such that $\tilde
  f\big|_{D_1} = f$ and for all $x,y\in \R^d$,
  \begin{align*} \|\tilde f(x)-\tilde f(y)\|_2 \le L\|x-y\|_2.
  \end{align*}
\end{thm}

\begin{remark}\label{rem convex} In Theorem~\ref{thm kirsz}, $\overline{\mathrm{conv}(D_2)}$ denotes the closed
  convex hull of $D_2$. In particular, if $D_2$ is already closed and convex, then $\overline{\mathrm{conv}(D_2)}
  = D_2$. \qede
\end{remark}
Finally, denote by $\S_{\ell_1}$ the sphere in $\R^{d}_{+}$ with respect to the $\ell_1$ norm, that is,
\begin{align*} \S_{\ell_1} = \big\{x\in \R_+^d~ |~ \|x\|_1 = 1\big\}.
\end{align*}
For $k \in \N$, define $S_k$ to be the `projection' of the $\ell_1$ $k$-sphere in $\Z_+^d$ onto $\S_{\ell_1}$,
that is,
\begin{align*}
  S_k = \left\{x \in \S_{\ell_1}~ | ~ k x \in \Z_+^d\right\},
\end{align*}
cf.\ Figure~\ref{fig:Sk}.
\begin{figure}[htb!]  \centering
  \includegraphics{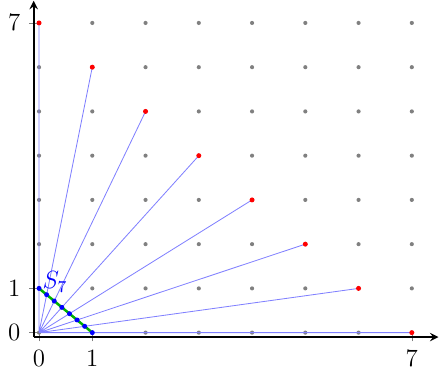}
  \caption{The discrete set $S_k$ (here for $k=7$) consists of the blue points sitting on the green line
    segment (and not of the red points).}
  \label{fig:Sk}
\end{figure}
Given $c\in\R$, $0<c<1$, define the set
\begin{align}
    \label{def D}
  \d_{c} = \big\{x\in \R_+^d~\big|~ \|x\|_1 = 1, ~c\, \e \le x \big\} = \mathbb
  S_{\ell_1}\cap\big(\R^{d}_{+}+c\e \big)
  \end{align}
  as depicted in Figure~\ref{fig:set-D}.
  \begin{figure}[htb!]  \centering
    \includegraphics{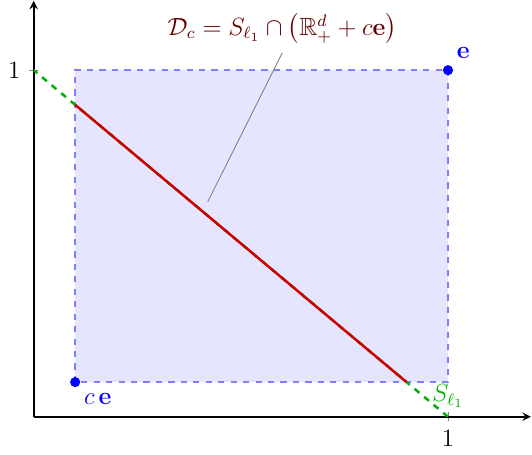}
    \caption{The set $\d_{c}$, here shown in dark red.}
    \label{fig:set-D}
  \end{figure}

  \begin{remark}\label{rem c k}
Note that if $c>1/d$ and $c\, e \le x$, then $\|x\|_1 \ge cd >1$. Therefore, $\d_{c}$ is not empty only when $c\le 1/d$. Also, note that if $k <d$ and $x\in S_k$, then $x$ must have at least one zero coordinate. Therefore, $S_k \cap \d_{c}$ is not empty only when $k \ge d$. \qede
\end{remark}

Next, we show that every point in $\d_{c}$ can be well approximated with a point in $S_k\cap \d_{c}$.

\begin{prop}\label{prop approx}
Assume that $\min\big\{k,\frac{1}{c}\big\} \ge d$. Then for every $x\in \d_{c}$, there exists $x_k \in S_k\cap \d_{c}$ such that
\begin{align*}
\|x-x_k\|_\infty \le \frac 2 k.
\end{align*}
\end{prop}

\begin{proof}
The set $\S_{\ell_1}$ is a simplex in $\R_+^d$ whose $\ell_\infty$ diameter is 1, i.e., the $\ell_\infty$ distance between any two points in $\S_{\ell_1}$ is at most 1, and it can be divided into $k^{d-1}$ simplexes, $C_1,\dots,C_{k^{d-1}}$, whose $\ell_\infty$ diameter is $1/k$, cf.\ Figure~\ref{fig:Simp}. The vertices of these simplexes are the points of $S_k$. Let $x\in \d_{c}$. Then $x$ belongs to a simplex $C_j$ for some $j\in \{1,\dots,k^{d-1}\}$, and $C_j$ has $\ell_\infty$ diameter $1/k$. If one of the vertices of $C_j$ lies inside $\d_{c}$, then since the diameter of $C_j$ is $1/k$, we found a point in $S_k\cap \d_{c}$ such that $\|x-x_k\|_\infty \le 1/k$. If this is not the case, then there must be a simplex which is adjacent to $C_j$ which has at least one vertex which lies in $\d_{c}$, cf.\ Figure~\ref{fig:Simp}.
\begin{figure}[htb!]  \centering
  \includegraphics{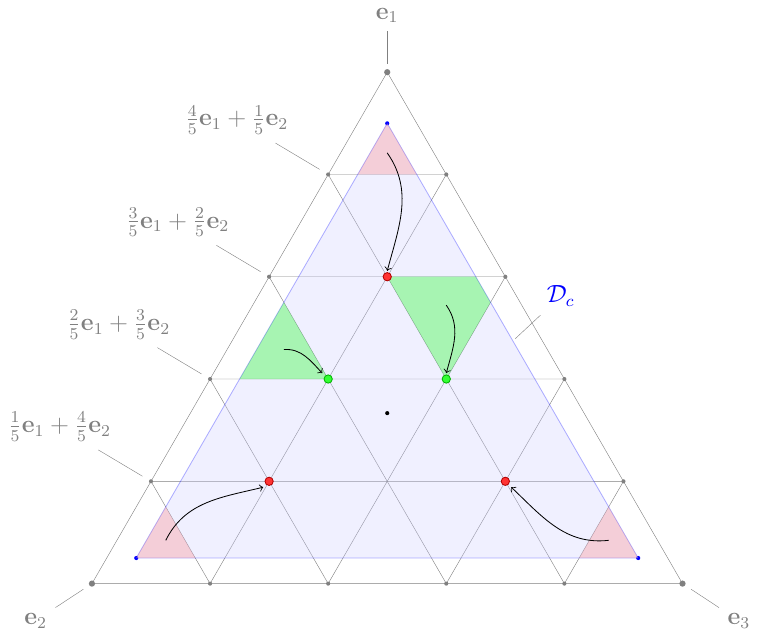}
  \caption{For the case $d=3$ the simplex $S_{\ell_{1}}$ is the convex
    hull of the standard unit vectors $\e_{1}$, $\e_{2}$, and $\e_{3}$ in
    $\R^{3}_{+}$. The points on the set $S_{k}$ (here for $k=5$) are
    shown in grey, and the set $\d_{c}$ is shown in blue.  While points in
    the green shaded simplexes intersected with $\d_{c}$ can be mapped to
    a vertex that is in $\d_{c}$, this is not the case for the ``corners''
    of the simplex $\d_{c}$ shown in red. The closest grid point for
    points $x$ near the vertices of $\d_{c}$ can be taken to be the
    opposing vertex of the neighboring simplex that shares the face
    intersecting $\d_{c}$. The distance to this vertex is always less than
    $2/k$ when the diameter of the simplexes $C_{j}$ is $1/k$.}
  \label{fig:Simp}
\end{figure}
Since the $\ell_\infty$ distance between any two points in adjacent simplexes is at most $2/k$, the result follows.
\end{proof}

We are now in a position to state and prove the main result.

\begin{thm}\label{main thm} Let $A:\Z_+^d\to \Z_+^d$ and $k \in \N$ is such that $k \ge d$. Assume that
  there exists $L\in [0,\infty)$ such that for all
  $x,y\in \Z_+^d$ with $\|x\|_1 = \|y\|_1 =k$,
  \begin{align}\label{contract prop}
    d_{\H}(Ax,Ay) \le L \, d_{\H}(x,y).
  \end{align}
  Assume also that there exists $c\in (0,1/d]$ such that for all $x\in \Z_+^d$ with $\|x\|_1 = k$,
  \begin{align}
    \label{eq:1}
    0 < c\|Ax\|_1 \e \le Ax.
  \end{align}
  Then there exists $y_k \in \Z_+^d$ with $\|y_k\|_1 = k$ satisfying
  \begin{align}\label{almost eigen} \left\| \frac{Ay_k}{\|Ay_k\|_1} - \frac{y_k}{\|y_k\|_1}\right\|_2 \le \frac
    {4Ld c^{-2} + 2\sqrt{d}}{k}.
  \end{align}
\end{thm}

Inequality~\eqref{almost eigen} says that when $\|y_{k}\|=k$ is large,
the vectors $y_k$ and $Ay_k$ are `almost' in the same direction,
making $y_k$ an `approximate' eigenvector of $A$, cf.,
Figure~\ref{figure-coneA}.

\begin{figure}[htb!]  \centering
  \includegraphics{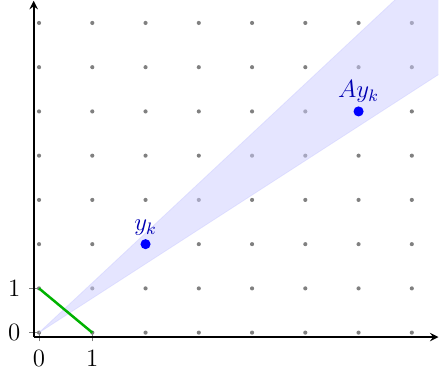}
  \caption{The point $y_k$ and its image $Ay_k$ lying in the same `narrow' cone.}
  \label{figure-coneA}
\end{figure}

\begin{proof}[Proof of Theorem~\ref{main thm}]
 Note first that if $d=1$, since $Ax \ge c\|Ax\|_1\e>0$, it follows that for every $y\in \N$, we have $Ay/\|Ay\|_1 = y/\|y\|_1 = 1$. Thus, the bound~\eqref{almost eigen} holds trivially.
 Assume then that $d\ge 2$. Let $B_k: S_k \to 
  \S_{\ell_{1}}
  $ be defined as follows,
  \begin{align}\label{def B} B_kx = \frac{A(kx)}{\|A(kx)\|_{1}},
  \end{align}
  where we note that $\|A(kx)\|_{1}\ne0$ due to~\eqref{eq:1}.
  Also, $B_k$ is well defined since $kx \in \Z_+^d$ whenever $x \in S_k$. It is known that the
  Hilbert metric~\eqref{def Hilbert} is invariant under dilation,
  cf.~\cite{LN12}*{Prop.~2.1.1},
  that is, if $x,y\in \R_+^d$ and
  $\alpha,\beta\in (0,\infty)$, then
  \begin{align}
    \label{dist invariance}
    d_{\H}(\alpha x,\beta y) = d_{\H}(x,y).
  \end{align}
  Assume that $x,y\in S_k$. We compute
  \begin{multline}
    \label{B is lip}
    d_{\H}(B_kx,B_ky) = d_{\H}\left(\frac{A(kx)}{\|A(kx)\|_{1}},
      \frac{A(ky)}{\|A(ky)\|_{1}}\right) \\
    \stackrel{\eqref{dist invariance}}{=} d_{\H}\big(A(kx), A(ky)\big)
    \stackrel{\eqref{contract prop}}{\le} L d_{\H}(kx, ky) \stackrel{\eqref{dist invariance}}{=} Ld_\H(x,y).
  \end{multline}
  By Remark~\ref{rem c k}, if we assume that $k \ge d$, then $S_k\cap \d_{c}\neq \emptyset$. Let $x,y\in S_k \cap \d_{c}$. Using~\eqref{equiv dist 2} and~\eqref{B is lip} with $\alpha =c$,
  $\beta =1$, we obtain
  \begin{align}\label{B euclid lip} \|B_kx-B_ky\|_2 \stackrel{\eqref{equiv dist 2}}{\le} \sqrt d \, d_\H(B_kx,
    B_ky) \stackrel{\eqref{B is lip}}{\le} L \sqrt d \, d_\H(x,y) \stackrel{\eqref{equiv dist 2}}{\le} \frac
    {2L\sqrt d}{c^2}\|x-y\|_2.
  \end{align}
  This means that $B_k$ is Lipschitz on $S_k\cap \d_{c}$ with respect
  to the Euclidean metric. Next, we study the invariance properties of
  $B_k$.  Since it was assumed that $c\|Ax\|_1 \e \le Ax$, we have
  that $B_k x \ge c\, \e$, hence $B_k(S_k) \subseteq \d_{c}$, and so
  $
  B_k\big(S_k \cap \d_{c} \big) \subseteq \d_{c}
  $
  .
  The set $\d_{c}$ is compact and convex since it is the intersection of two compact convex
  sets. By~\eqref{B euclid lip}, $B_k$ is Lipschitz on $S_k\cap \d_{c}$ with respect to the Euclidean
  metric. Therefore, by Theorem~\ref{thm kirsz} and Remark~\ref{rem convex}, it follows that there exists a map $
  \tilde B_k: \d_{c} \to \d_{c}$ such that for all $x,y \in \d_{c}$,
  \begin{align}\label{tilde B lip} \|\tilde B_k x - \tilde B_k y\|_2 \le \frac {2L\sqrt d}{c^2}\|x-y\|_2.
  \end{align}
  In particular, the map $\tilde B_k$ is a continuous map on a convex
  and compact set. Thus, by the Brouwer Fixed Point Theorem, there
  exists $\bar x_k\in \d_{c}$ such that
  $\bar x_k = \tilde B_k \bar x_k$.  By Proposition~\ref{prop approx},
  there exists $x_k\in S_k\cap \d_{c}$ such that
  $\|\bar x-x_k\|_\infty \le 2/k$. Since
  $\|x\|_2 \le \sqrt d\|x\|_\infty$ for all $x\in \R^d$, it follows
  that
  \begin{align}
    \label{dist to fixed pt}
    \|\bar x_k-x_k\|_2 \le \frac {2\sqrt{d}} {k}.
  \end{align}
  Therefore,
  \begin{align*} \left\|\frac{A(kx_k)}{\|A(kx_k)\|_1} - \bar x_k\right\|_2 = \|B_kx_k - \bar x_k\|_2
    \stackrel{(*)}{=} \|\tilde B_kx_k - \tilde B_k\bar x_k\|_2 \stackrel{\eqref{tilde B lip}}{\le} \frac {2L \sqrt d}{c^2}\|x_k-\bar x_k\|_2 \stackrel{\eqref{dist to fixed pt}}{\le} \frac {4 Ld}{c^2k},
  \end{align*}
  where in ($*$) we used the fact that $\tilde B_k = B_k$ on $S_k$ and $\tilde B_k \bar x_k =
  \bar x_k$. Altogether,
  \begin{align*}
  \left\|\frac{A(kx_k)}{\|A(kx_k)\|_1} - x_k\right\|_2 \le \left\|\frac{A(kx_k)}{\|A(kx_k)\|_1}
    - \bar x_k\right\|_2 + \|\bar x_k - x_k\|_2 \le \frac{4L d}{c^2k} + \frac {2\sqrt d} k = \frac{4Ldc^{-2}+2\sqrt d}{k}.
  \end{align*}
  Choosing $y_k = kx_k$ proves~\eqref{almost eigen} and this completes the proof of
  Theorem~\ref{main thm}.
\end{proof}

We make the following remarks regarding Theorem~\ref{main thm} and its proof.

\begin{remark}
  A particular class of maps that Theorem~\ref{main thm} and
  Corollary~\ref{cor sequence} (see below) apply to is the class of
  homogeneous maps. It is known that if $A:\R_+^d\to \R_+^d$ is
  $r$-homogeneous for some $r>0$, that is,
  $A(\rho x) = \rho ^r Ax$ for all $\rho> 0$, and monotone, that is,
  $Ax \le Ay$ whenever $x,y\in \R_+^d$ are such that $x\le y$, then
  $A$ satisfies $d_\H(Ax,Ay) \le r d_\H(x,y)$ for all $x,y\in \R_+^d$,
  cf., ~\cite{LN12}*{Cor.~2.1.4}. In fact, it is enough to assume
  that the map is $r$-subhomogeneous, that is,
  $A(\rho x) \le \rho^rAx$ for all $\rho \ge 0$. Clearly the same is
  true for integer maps. \qede
\end{remark}

\begin{remark}\label{rmk int} For every $k \in \N$, $x_k$ can be chosen to be in the set $\big\{x\in \R^d~|~
  c\, \e \le x \le \e\big\}$ (where $c\in(0,1/d]$ may or may not depend on $k$). This means $y_{k}$ can be chosen positive.\qede
\end{remark}

\begin{remark}
  Instead of the assumption that there exists $c \in (0,1/d]$ such that $Ax \ge c\|Ax\|_1\e$ for
  \emph{all} $x\in \Z_+^d$, it is enough to make the weaker assumption that $Ax \ge c\|Ax\|_1 \e$ whenever $x \ge
  c \|x\|_1 \e$. All that is really needed is the invariance of the set $\d_{c}$ as defined in~\eqref{def D}
  under the map $B_k$. \qede
\end{remark}

\begin{remark}\label{rem choice}
  The choice of the $\ell_1$-norm is essential in this proof. This is because we need the set $\d_{c}$ to be convex
  in order to use the Brouwer fixed-point theorem. This is different, e.g., from the proofs in~\cites{Koh82, Kra86},
  where any norm can be used.\qede
\end{remark}

\begin{remark}
  \label{rem:power-method}
  In many cases, one considers a map which is a contraction under the
  Hilbert metric, that is, $d_\H(Ax,Ay) < d_\H(x,y)$, and then
  proceeds to use the Banach contraction principle rather than the
  Brouwer fixed-point theorem. As a result, one obtains the ``power
  method'' for the computation of the Perron vector, i.e., that there
  exists $\overline{x} \in \R_+^d$, $\|\overline{x}\|=1$, such that for all $x\in \R_+^d\setminus\{0\}$,
  $A^nx/\|A^nx\| \stackrel{n \to \infty}{\longrightarrow}
  \overline{x}$. However, since in Theorem~\ref{main
    thm} Kirszbraun's extension theorem is used, the extended map is
  typically not a contraction and therefore this power method does not
  hold. \qede
\end{remark}

Next, it is shown that if we have good bounds on the norm of $Ax$ for
$x\in \Z_+^d$, then we can find a sequence
$\{y_k\}_{k=1}^{\infty}\subseteq \Z_+^d$ with a controlled behavior in
the following sense.

\begin{cor}\label{cor sequence}
  Let $A:\Z_+^d\to \Z_+^d$ and $k \in \N$ is such that $k \ge d$. Assume that
  there exists $L\in [0,\infty)$ such that for all $x,y\in \Z_+^d$
  with $\|x\|_1 = \|y\|_1 = k$,
  \begin{align*}
    d_{\H}(Ax,Ay) \le Ld_{\H}(x,y).
  \end{align*}
  Let us assume that there exists $c\in (0,1/d]$ such that for all $x\in \Z_+^d$ with $\|x\|_1 = k$,
  \begin{align*}
    0<c\|Ax\|_1 \e \le Ax.
  \end{align*}
  Assume also that there exists $a\in (0,\infty)$ such that $\|Ax\|_1 \ge a \|x\|_1$ for all $x\in
  \Z_+^d$ with $\|x\|_1 = k$. Then there exists $y_k \in \Z_+^d$ with $\|y_k\|_1=k$ such that
  \begin{align}
    \label{bound with a}
    a\left(1-\frac {4Ld c^{-2} + 2\sqrt{d}}{ck} \right)y_k \le Ay_k.
  \end{align}
  Alternatively, assume that there exists $b \in (0,\infty)$ such that $\|Ax\|_1 \le b \|x\|_1$ for
  all $x\in \Z_+^d$ with $\|x\|_1 = k$. Then there exists $y_k \in \Z_+^d$ with $\|y_k\|_1=k$ such that
  \begin{align}
    \label{bound with b}
    Ay_k \le b\left(1+ \frac {4Ld c^{-2} + 2\sqrt{d}}{ck} \right)y_k .
  \end{align}
  In both cases, we have $\|y_k\|_1 = k$. In particular, if $Ax \ge c\|Ax\|_1\e$ and $\|Ax\|_1 \ge a
  \|x\|_1$ ($0<\|Ax\|_1 \le b\|x\|_1$) for all $x\in \Z_+^d\setminus\{0\}$, then there exists a sequence $\{y_k\}_{k=1}^{\infty}\subseteq \Z_+^d$ such
  that for every $\eps >0$, there exists $k_0 \in \N$ such that for every $k \ge k_0$, $Ay_k \ge a (1-\eps)y_k$
  (respectively, $Ay_k \le b(1+\eps)y_k$).
\end{cor}

\begin{proof}
  By Theorem~\ref{main thm},  there exists $y_k \in \Z_+^d$ with $\|y_k\|_1 = k$ such that $\|Ay_{k}\|_{1}\ne0$ and
  \begin{align*}
  \left\|\frac{Ay_k}{\|Ay_k\|_1} - \frac{y_k}{\|y_k\|_1}\right\|_\infty \le \left\|\frac{Ay_k}{\|Ay_k\|_1} - \frac{y_k}{\|y_k\|_1}\right\|_2 \le \frac{4Ldc^{-2}+2\sqrt d}{k}.
  \end{align*}
  In particular, it follows that
  \begin{align*}
  - \left(\frac{4Ldc^{-2}+2\sqrt d}{k}\right)\e \le \frac{Ay_k}{\|Ay_k\|_1} - \frac{y_k}{\|y_k\|_1} \le \left(\frac{4Ldc^{-2}+2\sqrt d}{k}\right)\e,
  \end{align*}
  and so
  \begin{align*}
    \|A(y_k)\|_{1}\left(\frac{y_k}{\|y_k\|_1}- \frac {4Ld c^{-2} + 2\sqrt{d}}{k}\, \e \right) \le A(y_k)
    \le \|A(y_k)\|_{1}\left(\frac{y_k}{\|y_k\|_1} + \frac {4Ld c^{-2} + 2\sqrt{d}}{k}\, \e\right).
  \end{align*}
  By Remark~\ref{rmk int}, if $x_k = y_k/\|y_k\|_1 = y_k/k$, then $c\, \e \le
  x_k \le \e$ and so $\e \le \frac{1}{c} x_k = \frac 1 {ck}y_k$. Hence,
  \begin{align*}
    \left(1-\frac {4Ld c^{-2} + 2\sqrt{d}}{ck}\right)\frac{\|A y_k\|_1 y_k}{\|y_k\|_1} \le A y_k \le
    \left(1+ \frac {4Ld c^{-2} + 2\sqrt{d}}{ck} \right)\frac{\|A y_k\|_1 y_k}{\|y_k\|_1}.
  \end{align*}
  Assume that $\|Ax\|_1 \ge a\|x\|_1$ for all $x\in \Z_+^d$ with $\|x\|_1=k$. Then
  \begin{align*}
    A y_k & \ge \left(1- \frac {4Ld c^{-2} + 2\sqrt{d}}{ck} \right)\frac{\|A y_k\|_1 y_k}{\|y_k\|_1} \ge
            a\left(1-\frac {4Ld c^{-2} + 2\sqrt{d}}{ck} \right) y_k,
  \end{align*}
  which proves~\eqref{bound with a}. Alternatively, assuming that $\|Ax\|_1 \le b \|x\|_1$ for all $x\in \Z_+^d$ with $\|x\|_1=k$,
  \begin{align*}
    Ay_{k} & \le \left(1+\frac {4Ld c^{-2} + 2\sqrt{d}}{ck} \right)\frac{\|A y_k\|_1 y_k}{\|y_k\|_1} \le
              b\left(1+ \frac {4Ld c^{-2} + 2\sqrt{d}}{ck} \right) y_k,
  \end{align*}
  which proves~\eqref{bound with b}, and completes the proof.
\end{proof}

\begin{remark}
  In the case we can choose $a=1$ for all $k \in \N$ or $b=1$ for all $k\in \N$, Corollary~\ref{cor sequence} gives a sequence
  $\{y_k\}_{k=1}^{\infty}\subseteq \Z_+^d$ on which $A$ is `almost monotone'. That is, if $a=1$, then there
  exists a sequence $\{y_k\}_{k=1}^{\infty}\subseteq \Z_+^d$ such that for every $\eps>0$, there exists $k_0 \in
  \N$ such that for all $k \ge k_0$, $(1-\eps)y_k \le Ay_k$. This is close to a monotonicity property $y_k \le
  Ay_k$. Similarly in the case where $b=1$, we obtain a sequence $\{y_k\}_{k=1}^{\infty}\subseteq \Z_+^d$ with
  $\|y_k\|_1 = k$ such that $Ay_k \le (1+\eps)y_k$ for $k$ sufficiently large. \qede
\end{remark}

\section{Concave maps on integer lattices}\label{sec concave}

Recall that a map $F:\R^d_+\to \R_+^d$ is said to be concave if for every $x,y\in \R_+^d$ and every $\alpha \in
[0,1]$,
\begin{align*}
  F(\alpha x+(1-\alpha)y) \ge \alpha Fx+(1-\alpha)Fy.
\end{align*}
The notion of concavity can also be studied in $\Z_+^d$, cf.~\cite{BG18}.

\begin{defin}
  A map $A:\Z_+^{d_{1}}\to \Z_+^{d_{2}}$ is said to be \emph{concave}
  if for every $x, x_1,\dots,x_n \in \Z_+^{d_{1}}$ and
  $m,m_1,\dots,m_n\in \N$ such that $mx = \sum_{i=1}^nm_ix_i$ and
  $m = \sum_{i=1}^nm_i$ we have
  \begin{align}\label{ineq conc}
    mAx \ge \sum_{i=1}^nm_i Ax_i.
  \end{align}
  A map $A:\Z_+^{d_{1}} \to \Z_+^{d_{2}}$ is said to be \emph{affine} if in~\eqref{ineq conc} we have an equality rather than an inequality.
\end{defin}

The following result follows directly from Theorem 1 and Corollary 2 in~\cite{PW86}. It shows that every
concave map on $\Z_+^d$ can be extended to a concave map on $\R_+^d$.

\begin{thm}[\cite{PW86}]
  \label{thm extension} Assume that $A:\Z_+^d\to \Z_+^d$ is concave. Then there exists $F:\R_+^d\to
  \R_+^d$ concave such that $F\big|_{\Z_+^d} = A$.
\end{thm}

As a result, the following holds.

\begin{prop}\label{prop contract} Assume that $A:\Z_+^d\to \Z_+^d$ is concave and that $x,y\in \Z_+^d$
  are such that $\|x\|_1 = \|y\|_1>0$. Then $A$ is $d_{\H}$-nonexpansive, that is
  \begin{align*}
    d_{\H}(Ax,Ay) \le d_{\H}(x,y).
  \end{align*}
\end{prop}

\begin{proof}
  By Theorem~\ref{thm extension} there exists $F:\R_+^d\to \R_+^d$ concave such that $Fx=Ax$ for
  all $x\in \Z_+^d$. Assume that $m,n \in \Z_+$ are such that $mx \le ny$. Since $\|x\|_1 = \|y\|_1$, it follows
  that $m \le n$. Hence, there exists $z\in \R_+^d$ such that
  \begin{equation}
    \label{eq:3}
    ny = mx +(n-m)z.
  \end{equation}
  Since $F:\R_+^d\to \R_+^d$ is
  concave and positive,
  \begin{align*}
    nFy \ge mFx + (n-m)Fz \ge m Fx.
  \end{align*}
  Since $F\big|_{\Z_+^d} = A$, it follows that $nAy \ge m Ax$. Thus, by~\eqref{def lambda},
  $\lambda(Ax,Ay) \ge \lambda(x,y)$. Similarly, $\lambda(Ay,Ax) \ge \lambda(y,x)$. Therefore, by the definition
  of the Hilbert metric~\eqref{def Hilbert}, $d_{\H}(Ax,Ay) \le d_{\H}(x,y)$, and this completes the proof.
\end{proof}

\begin{exa}
  In~\eqref{eq:3} one cannot assume $z$ to be in $\Z^{d}_{+}$. E.g.,
  for $x=(1,1,3)$ and $y=(2,2,1)$ we verify $\|x\|_1 = \|y\|_1=5$ and
  have $\lambda(x,y)=\min_{i}\frac{y_{i}}{x_{i}}=\frac{1}{3}$, so
  $mx\leq ny$ holds with $m=1$ and $n=3$.  Hence, $(n-m)=2$ and the
  vector $z$ defined by~\eqref{eq:3} has to be
  $\big(\frac{5}{2},\frac{5}{2},0\big)$.
\end{exa}

\begin{remark}\label{rmk max} Proposition~\ref{prop contract} remains true if $A$ is a supremum of concave
  maps. However, this does not mean that Proposition~\ref{prop contract} holds for \emph{convex}
  maps. Indeed, any convex map $F:\R_+^d\to \R_+^d$ can be written as $F = \sup_{t \in \mathcal T} F_t$, where
  $\{F_t\}_{t\in \mathcal T}$ is a family of affine---and in particular concave---maps. If $\|x\|_1 = \|y\|_1$
  and $mx \le ny$, then as before we get $ny = mx+(n-m)z$ for some $z \in \R_+^d$. However, there is no guarantee
  that $F_t z \ge 0$, that is, $F_t$ does not necessarily map $\R_+^d$ to $\R_+^d$. Therefore, in general, Proposition~\ref{prop contract} does not hold for convex maps. \qede
\end{remark}

Using Proposition~\ref{prop contract} and Remark~\ref{rmk max}, we conclude the following.

\begin{thm}\label{thm concave case} Assume that $A_t:\Z_+^d\to \Z_+^d$, $t \in\mathcal T$, are concave maps
  and $k\in \N$ is such that $k \ge d$. Assume also there exist $c \in (0,1/d]$ such that $0<c\|A_t x \|_1 \e\le A_t x$ for all $t \in
  \mathcal T$ and $x\in \Z_+^d$ with $\|x\|_1=k$. Let $A = \sup_{t \in \mathcal T}A_t$. Then there exists
  $y_k \in \Z_+^d$ with $\|y_k\|_1 = k$ such that
  \begin{align}\label{almost eigen concave} \left\|\frac{A y_k}{\|A y_k\|_1} - \frac{y_k}{\|y_k\|_1}\right\|_2
    \le \frac{4dc^{-2}+2\sqrt d}{k}.
  \end{align}
  If there exists $b \in (0,\infty)$ such that $\|Ax\|_1 \le b\|x\|_1$ for all $x\in \Z_+^d$ with $\|x\|_1=k$, Then there exists $y_k \in \Z_+^d$ with $\|y_k\|_1 = k$ such that
  \begin{align}\label{upper bound concave} Ay_k \le b\left(1 + \frac{4dc^{-2}+2\sqrt d}{ck}\right)y_k.
  \end{align}
  If there exists $a \in (0,\infty)$ such that $\|A_{t_0}x\|_1 \ge a\|x\|_1$ for some $t_0\in
  \mathcal T$ and all $x\in \Z_+^d$ with $\|x\|_1=k$, then there exists $y_k \in \Z_+^d$ with $\|y_k\|_1 = k$
  and
  \begin{align}\label{lower bound concave} Ay_k \ge a\left(1 - \frac{4dc^{-2}+2\sqrt d}{ck}\right)y_k.
  \end{align}
\end{thm}

\begin{proof}
  Since $A = \sup_{t \in \mathcal T}A_t$, by Remark~\ref{rmk max} it follows that $d_{\H}(Ax,Ay)
  \le d_{\H}(x,y)$ for all $x,y\in \Z_+^d$ with $\|x\|_1 = \|y\|_1=k$. Hence, Theorem~\ref{main thm} holds with
  $L=1$, which proves~\eqref{almost eigen concave}. Next, if $\|Ax\|_1 \le b\|x\|_1$ for all $x\in \Z_+^d$ with $\|x\|_1=k$, then using Corollary~\ref{cor sequence} proves~\eqref{upper bound
    concave}. Finally, if $\|A_{t_0}x\| \ge a \|x\|_1$ for some $t_0 \in \mathcal T$ and all $x\in \Z_+^d$, $\|x\|_1=k$, then $\|Ax\|_1 \ge
  \|A_{t_0}x\|_1 \ge a \|x\|_1$. Therefore, again by using Corollary~\ref{cor sequence},~\eqref{lower bound
    concave} follows, and this completes the proof.
\end{proof}

\begin{remark}
  The reader might wonder why the previous proof is not based on
  Banach's contraction principle, given the Lipschitz bound for $A$ is
  bounded by $1$. The reason is that due to the extension via
  Kirszbraun's extension theorem---which requires the Euclidean norm
  and hence constants from the norm equivalence get introduced---the
  Lipschitz constant for the resulting extended map is not necessarily
  bounded above by $1$ anymore. As a result, there seems little hope
  that without additional assumptions the result could be proven by
  appealing to Banach's contraction principle. \qede
\end{remark}

\begin{remark}
  \label{remark-referee-comment}
  The authors were made aware by a referee that a concave map
  $f\colon \Z^{d}\to\Z^{d}$ is nonexpansive under Thompson's metric,
  \begin{equation*}
    \label{def Thompson}
    d_{\T}(x,y) =
    \begin{dcases}
      -\log \min\big\{ \lambda(x,y),\lambda(y,x)\big\} & \text{for }\min\big\{\lambda(x,y),\lambda(y,x)\big\} >0, \\ \infty &  \text{otherwise,} \end{dcases}
  \end{equation*}
  and would like to thank the referee for pointing out the following details
  which provide a useful alternative to proving results like
  Theorem~\ref{thm concave case}.

  Indeed, if
  $d_{\T} (x, y) = - \log \lambda (x,y) = -\log \min_{i}
  \frac{y_{i}}{x_{i}} = \log \frac{m}{n}$, then $m\geq n$, as
  $d_{\T}(x,y) \geq 0$. Thus, $\frac{n}{m} x \leq  y$ and
  $\frac{n}{m} y \leq x$. Now concavity gives
  $nf(x) \geq f(nx) + (n - 1)f(0) \geq f(nx)$ and likewise
  $mf(y) \geq f(my)$.

  Recall that $nx \leq my$, so that if we write $my = nx+(m-n)z$, we
  deduce that $z\in\R^{d}_{+}$. As in the proof of
  Proposition~\ref{prop contract}, this implies that the extension
  $F\colon\R^{d}_{+}\to\R^{d}_{+}$ of $f$ given by Theorem~\ref{thm
    extension} satisfies $mF(y) \geq n(F(x) + (n - m)F(z) \geq nF(x)$,
  so that $mf(y) \geq nf(x)$. Thus we have
  $\lambda(f(x),f(y))\geq \frac{n}{m}$.
  Using $\frac{n}{m} y \leq x$ it can be shown in the same way that
  $\lambda(f(y),f(x)) \geq \frac{n}{m}$. Thus $d_{\T}(f(x),f(y)) \leq d_{T}(x,y)$ for all
  $x,y \in \Z^{d}$.

  Now as $\big(\text{int\,} \R_{+}^{d}, d_{\T}\big)$ is isometric to
  $\big(\R^{d}, \|\cdot\|_{\infty}\big)$ (component-wise $\log$ is an
  isometry, c.f.~\cite{LN12}*{Prop.~2.2.1}), and every sup-norm
  nonexpansive map on subset of $\R^{d}$ can be extended in a
  nonexpansive way to the whole $\R^{d}$, c.f.~\cite{LN12}*{Lem.~4.2.4}, we find that $f\colon \Z_{+}^{d}\to\Z_{+}^{d}$ can be
  extended into a $d_{\T}$-nonexpansive way to $\R^{d}_{+}$. More
  details can be found in \cite{LN12}*{Ch.~4}.~\qede
\end{remark}

It is known that a minimum of affine maps is concave (see~\cite{BG18}). As an immediate corollary of Theorem~\ref{thm concave case},
we obtain the following results for a `zigzag' map, that is, a map
which is a maximum of minima of a finite number of affine maps.

\begin{cor}\label{cor zigzag}
  Assume that $A_{i,j}:\Z_+^d\to \Z_+^d$, $i\in \{1,\dots,I\}$, $j\in \{1,\dots,J\}$ are affine maps, where $I,J\in \N$, and $k\in \N$ is such that $k \ge d$. Assume also that
  there exists $c\in (0,1/d]$ such that $0<c\|A_{i,j}x\|_1e \le A_{i,j}x$ for all $x\in \Z_+^d$ with $\|x\|_1 = k$ and for all $i\in \{1,\dots,I\}$,
  $j\in \{1,\dots,J\}$. Define $A:\Z_+^d \to \Z_+^d$ by
  \begin{align*}
    A = \max_{i \in \{1,\dots,I\}} \min_{j \in \{1,\dots, J\}} A_{i,j}.
  \end{align*}
  Then there exists $y_k \in \Z_+^d$ with $\|y_k\|_1 = k$ such that
  \begin{align*} \left\|\frac{A y_k}{\|A y_k\|_1} - \frac{y_k}{\|y_k\|_1}\right\|_2 \le \frac{4dc^{-2}+2\sqrt d}{k}.
  \end{align*}
  Also, if there exists $b\in (0,\infty)$ such that $0<\|Ax\|_1 \le b\|x\|_1$ for all $x\in \Z_+^d$ with $\|x\|_1=k$, then there exists $y_k \in \Z_+^d$ with $\|y_k\|_1 = k$ such that
  \begin{align*}
    Ay_k \le b\left(1 + \frac{4dc^{-2}+2\sqrt d}{ck}\right)y_k.
  \end{align*}
  If, in addition, there exists $a\in (0,\infty)$ such that $a\|x\|_1 \le \|\min_{j\in \{1,\dots,J\}}A_{i_0,j}x\|_1$ for
  some 
  $i_0\in \{1,\dots,I\}$ and all $x\in \Z_+^d$ with $\|x\|_1=k$, then there exists $y_k \in \Z_+^d$ with $\|y_k\|_1 = k$
  such that
  \begin{align*}
    Ay_k \ge a\left(1 - \frac{4dc^{-2}+2\sqrt d}{ck}\right)y_k.
  \end{align*}
\end{cor}

Next, it is shown that any concave map on $\Z_+^d$ has a controlled growth rate, that is, there is always $b\in (0,\infty)$ such that $\|Ax\|_1 \le b\|x\|_1$. This is similar to the case of concave
maps on $\R_+^d$.

\begin{prop}\label{prop bounded} Assume that $A:\Z_+^d\to \Z_+^d$ is concave. Then there
  exists $b' \in (0,\infty)$ such that for all $x\in \Z_+^d\setminus\{0\}$,
  \begin{align*}
  \|Ax\|_1 \le \big(\|A(0)\|_1 + b'd\big)
  \|x\|_1.
  \end{align*}
\end{prop}

\begin{proof}
  First, we would like to show that $Ax \ge A(0)$ for all $x\in \Z_+^d$. Indeed, write $A = (A_1,\dots, A_d)$, where $A_i:\Z_+^d \to \Z_+$, $i \in \{1,\dots,d\}$, are concave. Assume that we dot not have $Ax \ge A(0)$ for all $x\in \Z_+^d$.
  Then without loss of generality assume that $A_1 x < A_1(0)$ for some $x\in \Z_+^d$. Since for all $k \in \N$, $kx = 1\cdot(kx)+(k-1)\cdot 0$ and since $A_1$ is concave,
\begin{align*}
A_1(kx) \le k(A_1x - A_1(0)) + A_1(0).
\end{align*}
If $k$ is sufficiently large, then we must have $A_1(kx)<0$ which is a contradiction to the assumption that $A$ is positive. This also means that the map $A - A(0)$ is a concave map and $Ax-A(0) \in \Z_+^d$ for all $x\in \Z_+^d$.
  By Theorem~\ref{thm extension}, there exists $F:\R_+^d\to \R_+^d$ concave such that
  $F\big|_{\Z_+^d} = A-A(0)$. Note that $F(0)=0$. It is known that in such case there exists $b'\in (0,\infty)$ such that $F(x/\|x\|_{1})
  \le b'\, \e$, cf., e.g.,~\cite{Kra86}*{p.~281}. Let $x\in \Z_+^d\setminus \{0\}$. Then we must have $\|x\|_1
  \ge 1$. Hence, by the concavity of $F$ and the fact that $F(0) = A(0) - A(0) =0$,
  \begin{align*}
    F\left(\frac x{\|x\|_1}\right) \ge \frac 1 {\|x\|_1} Fx + \left(1-\frac 1 {\|x\|_1}\right)F(0)
    = \frac{Fx}{\|x\|_1}.
  \end{align*}
  Hence, we have $Fx \le b' \|x\|_1\e$. Since $Fx = Ax- A(0)$ for all $x\in \Z_+^d$ and since $\|\e\|_1 =d$, it follows that
  \begin{align*}
  \|Ax\|_1 \le \|A(0)\|_1 + \|Fx\|_1 \le \|A(0)\|_1 + b'd\|x\|_1 \le \big(\|A(0)\|_1+b'd\big)\|x\|_1,
  \end{align*}
  which completes the proof.
\end{proof}

\begin{remark}
Proposition~\ref{prop bounded} is by no means optimal. In particular, for some $k\in \N$, we might find $b\in (0,\infty)$ much smaller than $\|A(0)\|_1+b'd$ such that $\|Ax\|_1 \le b\|x\|_1$ for all $x\in \Z_+^d$ with $\|x\|_1=k$. \qede
\end{remark}

Finally, we show if the image of a concave map on $\R_+^d$ is always rounded up to the next integer value, one
obtains an integer map which has a controlled Lipschitz constant. Here for $x = (x_1,\dots,x_d)\in \R_+^d$,
denote by $\lceil x\rceil$ the vector of rounded-up integer values, that is,
$\big(\lceil x_1\rceil, \dots, \lceil x_d\rceil\big)$, where $\lceil r\rceil$ denotes the smallest integer
that is larger than or equal to the real number $r$.  This will be of use when we study the applications in
Section~\ref{sec app}.

\begin{prop}\label{prop ceil} Assume that $F:\R_+^d \to \R_+^d$ is concave and $k\in \N$. Assume also that there exist $a\in
  (0,\infty)$ and $c\in (0,1/d]$ such that $Fx \ge c\|Fx\|_1 \e>0$ and $\|Fx\|_1 \ge a\|x\|_1$ for all $x\in
  \R_+^d$ with $\|x\|_1=k$. Let $Ax = \lceil Fx\rceil$. Then for all $x,y \in\Z_+^d$ with $\|x\|_1 = \|y\|_1 = k$,
  \begin{align*}
    d_\H (Ax, Ay) \le \left(1+\frac 2 {ac}\right)d_\H(x,y).
  \end{align*}
\end{prop}

\begin{proof}
  Assume that $x,y \in \Z_+^d$ are such that $\|x\|_1 = \|y\|_1 = k$ and $m,n \in \Z_+$ are such that
  $mx \le ny$. Then in particular it follows that $m \le n$. As before, since $F$ is concave on $\R_+^d$, there
  exists $z \in \R_+^d$ such that $ny = mx + (n-m)z$ and as a result $nFy \ge mFx + (n-m)Fz \ge mFz$. Taking the
  ceiling function and using the fact that it is known that for every $x\in \R_+^d$, $x \le \lceil x\rceil \le
  x+\e$, gives $nAy \ge m(Ax-\e)$. Since $Fx \ge c\|Fx\|_1 \e$, it follows that
  \begin{align}\label{almost lip} \frac m n Ax \le Ay + \frac m n \e \stackrel{(*)}{\le} Ay + \frac{Fy}{c\|Fy\|_1}
    \stackrel{(**)}{\le} Ay + \frac{Ay}{ac\|y\|_1} = \left(1+ \frac{1}{ac\|y\|_1}\right)Ay,
  \end{align}
  where in ($*$) we used the fact that $m \le n$ and the assumption on $F$, and in ($**$) we used the fact that $Fy \le Ay$
  and $\|Fy\|_1 \ge a \|y\|_1$. Altogether, by~\eqref{def lambda} combined with~\eqref{almost lip}, it follows
  that
  \begin{align*} \lambda(Ax, Ay) \ge \frac{\lambda(x,y)}{1+ \frac 1 {ac\|y\|_1}}.
  \end{align*}
  Similarly,
  \begin{align*} \lambda(Ay, Ax) \ge \frac{\lambda(y,x)}{1+ \frac 1 {ac\|x\|_1}}.
  \end{align*}
  Therefore, by~\eqref{def Hilbert}, it follows that
  \begin{align}\label{bound with log} \nonumber d_\H(Ax,Ay) & \le  d_\H(x,y) + \log\left(1+ \frac 1
                                                                       {ac\|x\|_1}\right) + \log\left(1+ \frac 1 {ac\|y\|_1}\right) \\ & \stackrel{(*)}{\le}  d_\H(x,y) + \frac 1 {ac
                                                                                                                                                               \|x\|_1} + \frac 1 {ac\|y\|_1},
  \end{align}
  where in ($*$) we used the fact that $\log(1+x) \le x$ for all $x \in \R_+$. Now, since
  $\|x\|_1 = \|y\|_1$, it follows that $\max\{x,y\} \le \|x\|_1\e$. Also, since $x,y \in \Z_+^d$, it follows that
  $\|x-y\|_{\infty} \ge 1$ whenever $x\neq y$. Therefore, by the right-hand inequality in~\eqref{equiv dist infty},
  $d_\H(x,y) \ge 1/\|x\|_1$ and also $d_\H(z,y) \ge 1/\|y\|_1$. Using this in~\eqref{bound with log}, it follows
  that
  \begin{align*}
    d_\H(Ax,Ay) \le \left(1+\frac 2 {ac}\right)d_\H(x,y),
  \end{align*}
  and this completes the proof.
\end{proof}

\section{Applications}\label{sec app}

\subsection{A discrete epidemic model}

One of the oldest epidemic models is the
\emph{Susceptible-Infected-Susceptible} (SIS) model, which is a
special case of the model studied in \cite{KM27} and describes the
infection rates in a system with several separate locations, say,
different cities or countries. The continuous version of this model
can be described by the following system of differential equations for
$d$ different locations. For $i \in \{1,\dots,d\}$ let
$x_i(t) \in [0,1]$ denote the portion of population at location $i$
which is infected at time $t \in \R_+$. Then $x_{i}(t)$ changes
according to
\begin{align}
  \label{sis cont}
  \dot x_i(t) = -\delta_i x_i(t) + \sum_{j=1}^db_{i,j}\big[x_j(t)(1-x_i(t))\big],
\end{align}
where $\delta_i \in [0,1]$ and $b_{i,j} \ge 0$, $i,j \in \{1,\dots,d\}$ are model parameters, cf.,
e.g.,~\cite{NPP16} for more information about this and other epidemic models.

Assume now that $x_i \in \Z_+$ is the number of infected people at location $i$, which has a total population
of $M_i$. Then a discrete version of~\eqref{sis cont} would be
\begin{align*} \frac{x_i(n+1)}{M_i} - \frac{x_i(n)}{M_i} = -\delta_i \frac{x_i(n)}{M_i} +
  \sum_{j=1}^db_{i,j}\frac{x_j(n)}{M_j}\left(1-\frac{x_i(n)}{M_i}\right),
\end{align*}
which gives
\begin{align*}
  x_i(n+1) = \delta_i' x_i(n) + \sum_{j=1}^db_{i,j}\frac{x_j(n)}{M_j}\big(M_i-x_i(n)\big),
\end{align*}
with $\delta_i' = 1-\delta_i \in [0,1]$. In the discrete setting, it is therefore natural to
consider the difference equation $x(n+1) = A(x(n))$ where $n \in \Z_+$, and for $i \in \{1,\dots, d\}$,
$A_i:\Z_+^d\to \Z_+$ is given by
\begin{align}
  \label{def A i}
  A_ix = \min\big\{ M_i, \lceil P_ix\rceil \big\},
\end{align}
with $P_i : \R^d\to \R$ being the (approximate) infected population, which is given by
\begin{align}
  \label{def B i}
  P_ix = \delta_i' x_i + \sum_{j=1}^db_{i,j}\frac{x_j}{M_j}\big(M_i-x_i\big),
\end{align}
where $\delta_i', b_{i,j} \in \R_+$. The choice of a ceiling function in~\eqref{def A i} rather than a floor
function is not particularly important. It only makes some of the calculations below slightly simpler.

Next, it is shown that under certain assumptions on the coefficients of the the operators $P_i$, we can obtain
a Lipschitz condition on the map $A$. 

\begin{prop}\label{prop SIS} Let $A=(A_1,\dots,A_d)$ is the map defined in~\eqref{def A i}. Assume that there exist numbers $\delta_*, \delta^*, B_*, B^*, R_*, R^* \in
  (0,\infty)$ such that for all $i,j\in \{1,\dots,d\}$,
  \begin{alignat}{2}
    \label{bounds del}  \delta_* & \le \; \delta_i' &&\le \delta^*, \\
    \label{bounds B}  B_* & \le  b_{i,j} &&\le B^*,  \\
    \label{bounds R}  R_* & \le\dfrac{M_i}{M_j} &&\le R^*.
  \end{alignat}
  Assume that $k \in \N$ is such that
  \begin{align}\label{large bound k} k \le \min\{M_1,\dots,M_d\}.
  \end{align}
  Then for every $x = (x_1,\dots, x_d) \in \Z_+^d$ with $\|x\|_1 = k$ and $x_i \le M_i$ for all $i
  \in \{1,\dots,d\}$,
  \begin{align}\label{bound A} \frac 1 2 \min\left\{\delta_*, R_*B_*\right\} k \e \le Ax \le \left(\delta^* +
    R^*B^* + 1\right)k\e.
  \end{align}
  If it is assumed further that $k \in \N$ is such that
  \begin{align}\label{bounds k} k \le \frac{\min\{M_1,\dots,M_d\}}{\delta^* + R^*B^*+1}.
  \end{align}
  Then for every $x,y \in \Z_+^d$ with $\|x\|_1 = \|y\|_1 = k$ and $x_i, y_i \le M_i$ for all $i
  \in \{1,\dots,d\}$,
  \begin{align}\label{lip cond A} d_\H(Ax,Ay) \le \frac{4\left(\delta^*+R^*B^*+1\right)\left(\delta^* +
    R^*B^*d+B^*d+2\right)}{\min\left\{\delta_*^2, R_*^2B_*^2\right\}}\, d_\H(x,y).
  \end{align}
\end{prop}

\begin{proof}
  Fix $i \in \{1,\dots,d\}$. If $x_i \ge \frac k 2$, then by~\eqref{def B i}, $P_ix \ge \delta_i'
  x_i \ge \frac{1}{2} \delta_*k$, and so $\lceil P_ix\rceil \ge \lceil \frac 1 2 \delta_*k \rceil \ge \frac 1 2
  \delta_* k$. Alternatively, if $x_i \le \frac k 2$, then by~\eqref{large bound k} it follows that $1 -
  \frac{x_i}{M_i} \ge \frac 1 2$, and so
  \begin{align*}
    P_ix \ge \sum_{j=1}^db_{i,j}\frac{x_j}{M_j}\big(M_i-x_i\big) =
    \sum_{j=1}^db_{i,j}x_j\frac{M_i}{M_j}\left(1-\frac{x_i}{M_i}\right) \stackrel{\eqref{bounds R}}{\ge} \frac 1 2
    R_*\sum_{j=1}^db_{i,j}x_j \stackrel{\eqref{bounds B}}{\ge} \frac 1 2R_* B_* k.
  \end{align*}
  Thus, $\lceil P_ix \rceil \ge \lceil \frac 1 2 R_*B_*k\rceil \ge \frac 1 2 R_*B_*k$. In both
  case, we obtain $\lceil P_ix\rceil \ge \frac 1 2 \min\big\{\delta_*, R_*B_*\big\} k$. Since $\delta_* \le 1$,
  it follows that $M_i \ge \frac 1 2 \min\left\{\delta_*, R_*B_*\right\}$ for all $i \in \{1,\dots,d\}$, and so
  by the definition of $A_i$ in~\eqref{def A i}, $A_ix \ge \frac 1 2 \min\left\{\delta_*, R_*B_*\right\}$. This
  proves the left-hand inequality in~\eqref{bound A}.  On the other hand, for every $i \in \{1,\dots,d\}$, we have
  $P_ix \le \delta^* x_i + R^*B^* k \le \left(\delta^*+R^*B^*\right)k$. Therefore,
  \begin{align}\label{upper bound A i} A_ix \le \lceil P_i x\rceil \le P_ix + 1 \le
    \left(\delta^*+R^*B^*\right)k + 1 \le \left(\delta^*+R^*B^*+1\right)k,
  \end{align}
  and so $Ax \le \left(\delta^*+R^*B^*+1\right)k \e$, which proves the right-hand inequality
  in~\eqref{bound A}.

  Next, let $x,y \in \Z_+^d$ with $\|x\|_1 = \|y\|_1 = k$, and assume that $k \le \min\{M_1,\dots,M_d\}$ and
  $x_i,y_i \le M_i$ for all $i \in \{1,\dots,d\}$. Then by Proposition~\ref{prop equiv} combined
  with~\eqref{bound A}, it follows that
  \begin{align}\label{from hil to infty} \frac{\min\left\{\delta_*^2,
    R_*^2B_*^2\right\}k}{4\left(\delta^*+R^*B^*+1\right)}d_\H(Ax,Ay) \le \|Ax-Ay\|_{\infty}.
  \end{align}
  To estimate $\|Ax-Ay\|_\infty$, note that for every $i \in \{1,\dots, d\}$,
  \begin{align*}
    P_ix-P_iy & = \delta_i'(x_i-y_i) + \sum_{j=1}^db_{i,j}\frac{x_j}{M_j}\big(x_i-M_i\big) -
                \sum_{j=1}^db_{i,j}\frac{y_j}{M_j}\big(y_i-M_i\big) \\
              & = \delta_i'(x_i-y_i) +
                \sum_{j=1}^db_{i,j}(x_j-y_j)\frac{M_i}{M_j}\left(1-\frac{y_i}{M_i}\right) + \sum_{j=1}^db_{i,j}x_j
                \frac{M_i}{M_j}\left(\frac{y_i}{M_i} - \frac{x_i}{M_i}\right) \\
              & = \delta_i'(x_i-y_i) +
                \sum_{j=1}^db_{i,j}(x_j-y_j)\frac{M_i}{M_j}\left(1-\frac{y_i}{M_i}\right) + \sum_{j=1}^db_{i,j}
                \frac{y_j}{M_j}\left(y_i - x_i\right).
  \end{align*}
  Since $M_i/M_j \le R^*$, $1-y_i/M_i \le 1$, and $y_j/M_j \le 1$, it follows that
  \begin{align*}
    |P_ix - P_iy| & \le \delta^*|x_i-y_i| + R^*\sum_{j=1}^db_{i,j}|x_j-y_j|+
                    \sum_{j=1}^db_{i,j}|x_j-y_j| \\
                  & \le \left(\delta^* + R^*B^*d+B^*d\right)\|x-y\|_{\infty},
  \end{align*}
  which implies
  \begin{align*}
    \big|\lceil P_ix \rceil - \lceil P_iy\rceil \big| & \le  |P_ix - P_iy| + 2 \\
                                                      & \stackrel{(*)}{\le}  \left(\delta^* + R^*B^*d+B^*d\right)\|x-y\|_{\infty} + 2\|x-y\|_\infty \\
                                                      & =  \left(\delta^* + R^*B^*d+B^*d+2\right)\|x-y\|_{\infty},
  \end{align*}
  where in ($*$) we used the fact that since $x,y \in \Z_+^d$, $\|x-y\|_{\infty} \ge 1$
  whenever $x\neq y$ (the case $x=y$ is trivial). Now, by~\eqref{upper bound A i} and the choice of
  $k$~\eqref{bounds k}, it follows that $Ax = \lceil Px \rceil$, $Ay = \lceil Py\rceil$. Thus,
  \begin{align}\label{bound infty norm} \|Ax-Ay\|_{\infty} \le \left(\delta^* +
    R^*B^*d+B^*d+2\right)\|x-y\|_{\infty}.
  \end{align}
  Again by Proposition~\ref{prop equiv}, if $\|x\|_1 = \|y\|_1 = k$
  then in particular $\max\{x,y\} \le k \e$ and so
  \begin{align}\label{bound infty to hil} \|x-y\|_\infty \le k d_\H(x,y)
  \end{align}
  Altogether,
  \begin{align*}
    d_\H(Ax,Ay) & \stackrel{\eqref{from hil to infty}}{\le} 
                                                              \frac{4\left(\delta^*+R^*B^*+1\right)}{\min\left\{\delta_*^2, R_*^2B_*^2\right\}k} \|Ay-Ax\|_{\infty} \\ &
                                                                                                                                                                         \stackrel{\eqref{bound infty norm}}{\le}  \frac{4\left(\delta^*+R^*B^*+1\right)\left(\delta^* +
                                                                                                                                                                                                                    R^*B^*d+B^*d+2\right)}{\min\left\{\delta_*^2, R_*^2B_*^2\right\}k}\|y-x\|_{\infty} \\ & \stackrel{\eqref{bound
                                                                                                                                                                                                                                                                                                            infty to hil}}{\le}  \frac{4\left(\delta^*+R^*B^*+1\right)\left(\delta^* +
                                                                                                                                                                                                                                                                                                                                  R^*B^*d+B^*d+2\right)}{\min\left\{\delta_*^2, R_*^2B_*^2\right\}}d_\H(y,x),
  \end{align*}
  which proves~\eqref{lip cond A} and completes the proof.
\end{proof}

Applying Theorem~\ref{main thm} and Corollary~\ref{cor sequence}, we obtain the following.

\begin{cor}
  Assume that $k \in \N$ satisfies
  \begin{align*}
    k \le \frac{\min\{M_1,\dots,M_n\}}{\delta^* + R^*B^*+1}.
  \end{align*}
  Then Theorem~\ref{main thm} and Corollary~\ref{cor sequence} hold with
  \begin{align*}
    L = \frac{4\left(\delta^*+R^*B^*+1\right)\left(\delta^* +
    R^*B^*d+B^*d+2\right)}{\min\left\{\delta_*^2, R_*^2B_*^2\right\}},
  \end{align*}
  and
  \begin{align*}
    a= \frac 1 2 \min\left\{\delta_*, R_*B_*\right\}d, \quad b = \left(\delta^* + R^*B^* +
    1\right)d, \quad c = \frac{\min\left\{\delta_*, R_*,B_*\right\}}{2d\left(\delta^* + R^*B^* + 1\right)}.
  \end{align*}
\end{cor}

\begin{proof}
  The choice of $L$ follows from Proposition~\ref{prop SIS}. Next, assume that $x\in \Z_+^d$ is
  such that $\|x\|_1 = k$. By taking the $\ell_1$-norm in~\eqref{bound A} and using the fact that $\|x\|_1 = k$
  and $\|\e\|_1 = d$, it follows that
  \begin{align*} \frac 1 2 \min\left\{\delta_*, R_*,B_*\right\}d\,\|x\|_1 \le \|Ax\|_1 \le \left(\delta^* +
    R^*B^* + 1\right) d\,\|x\|_1.
  \end{align*}
  Hence, now using the left-hand inequality in~\eqref{bound A},
  \begin{align*}
    Ax \ge \frac 1 2 \min\left\{\delta_*, R_*,B_*\right\} k \e \ge \frac{\min\left\{\delta_*,
    R_*,B_*\right\}}{2d\left(\delta^* + R^*B^* + 1\right)}\|Ax\|_1 \e.
  \end{align*}
  The claim now follows.
\end{proof}

\begin{exa} 
  To consider a concrete example, assume that there are $d=3$
  locations (e.g., countries). Assume also that all
  three locations have the same population, that is,
  $M_1 = M_2 = M_3 = M$. This means in particular that
  $R_* = R^* = 1$. Assume also that $\delta_* = B_* = 1/2$ and
  $\delta^* = B^* = 3/4$. In such case,
  \begin{align*}
    L = \frac{4\left(\delta^*+R^*B^*+1\right)\left(\delta^* +
    R^*B^*d+B^*d+2\right)}{\min\left\{\delta_*^2, R_*^2B_*^2\right\}} = 290,
  \end{align*}
  and
  \begin{align*}
    a = \frac 1 2 \min\left\{\delta_*, R_*B_*\right\}d = \frac 3 4, \quad  b = \left(\delta^* + R^*B^* +
    1\right)d = \frac{30}{4}, \quad  c = \frac{\min\left\{\delta_*, R_*B_*\right\}}{2d\left(\delta^* + R^*B^* +
    1\right)} = \frac1 {30}.
  \end{align*}
  Altogether,
  \begin{align*}
    4Ld c^{-2} + 2\sqrt{d} \approx 3,132,003,
  \end{align*}
  and
  \begin{align*}
  \sqrt d \left(4Ld c^{-2} + 2\sqrt{d}\right) \approx 5,424,789.
  \end{align*}
  Hence, for every $k > 5,424,789$, there exists $y_k \in \{1,\dots,M\}^3$ such that
  \begin{align*} \left\|\frac{A y_k}{\|A y_k\|_1} - \frac{y_k}{\|y_k\|_1}\right\|_1 \le \sqrt d\left\|\frac{A y_k}{\|A y_k\|_1} - \frac{y_k}{\|y_k\|_1}\right\|_2 \le \frac{5,424,789}k <1.
  \end{align*}
  Note that the approximation becomes truly efficient only when $M$ is very large.
  In order to make use of Corollary~\ref{cor sequence}, since the error term there is
  $\frac{4Ld c^{-2} + 2\sqrt{d}}{c} \approx 93,960,104$, we need to choose $k > 93,960,104$ in order to get an error which is smaller than~1.
\end{exa}

\subsection{Additive Increase Multiplicative Decrease model}
The \emph{Additive Increase Multiplicative Decrease} (AIMD) model is
an algorithm for negotiating in a decentralized fashion a fair share
of a limited resource among several entities. A typical example is the
allocation of bandwidth among different users in the \emph{transmission
control protocol} (TCP), which is used in essentially every internet
capable device nowadays.

The AIMD model was first introduced in~\cite{CJ89}, cf., also the recent monograph~\cite{CKSW16}. In this model, users increase
their demand (transmission rate in the case of TCP) by a fixed additive amount, until they receive a message (from a central router in the case of TCP) that global
capacity has been reached, in which case they decrease their demand by a multiplicative factor. To formulate the model
more precisely, assume that there are $d \ge 1$ users and denote by $x_i(t)$ the share at time $t \in \Z_+$ of the $i$th
user. Also, denote by $k \in \N$ the global capacity of the resource available to all users. Therefore, for all
$t \in \Z_+$, $\sum_{i=1}^dx_i(t) \le k$. If $x(t)$ denotes the vector $(x_1(t),\dots,x_d(t))$, then the capacity
requirement can be written as $\|x(t)\|_1 \le k$. For $n \in \Z_+$, let $t_n$ denote the times when utilization reaches
total capacity, that is, $\sum_{i=1}^dx_i(t_n) = k$ or $\|x(t_n)\|_1 = k$. Denote also $x_i(n) = x _i(t_n)$. In the
continuous case, the simplest AIMD model is described by the following system of equations,
\begin{align}\label{AIMD cont lin} x_i(n+1) = \alpha_i x_i(n) + \beta_i T(n), \quad  i \in \{1,\dots,d\},
\end{align}
where for all $i \in \{1,\dots,d\}$, $\alpha_i \in [0,1)$ and $\beta_i >0$, and $T(n) =
t_{n+1}-t_n$. One can consider a nonlinear version of~\eqref{AIMD cont lin}, as follows,
\begin{align}\label{AIMD cont nonlin} x_i(n+1) = A_i(x_i(n)) + B_i (T(n)), \quad  i \in \{1,\dots,d\},
\end{align}
where now $A_i, B_i: \R_+\to \R_+$. Such nonlinear versions were studied in~\cites{CS12, KSWA08,
  RS07}. As a discrete version of~\eqref{AIMD cont nonlin}, we can consider the following system of equations,
\begin{align}\label{AIMD disc} x_i(n+1) = \big\lceil A_i(x_i(n)) + B_i (T(n))\big\rceil, \quad  i \in
  \{1,\dots,d\},
\end{align}
where now $A_i,B_i : \Z_+ \to \R_+$. Using Proposition~\ref{prop ceil}, the following holds.

\begin{prop}
  Assume that $A = (A_1,\dots,A_d) : \R_+^d \to \R_+^d$ is concave and $B = (B_1,\dots,B_d):\R_+^d
  \to \R_+^d$ and $k\in \N$ is such that $k \ge d$. Assume also that there exist $a \in (0,\infty)$ and $c\in (0,1/d]$ such that $Fx \ge c\|Fx\|_1\e>0$ and
  $\|Fx\|_1 \ge a\|x\|_1$ whenever $\|x\|_1 = k$, where $F = (F_1,\dots,F_d): \R_+^d \to \R_+^d$ is given by
  \begin{align*}
    F_i(x) = A_i(x_i) + B_i(T(n)).
  \end{align*}
  Then for capacity $k$ and time $n \in \Z_+$, there exists a configuration $x(n) \in
  \Z_+^d$ such that
  \begin{align} \left\|\frac{x(n+1)}{\|x(n+1)\|_1} - \frac{x(n)}{\|x(n)\|_1}\right\|_2 \le
    \frac{4\left(1+\frac 1 {ac}\right)dc^{-2}+2\sqrt d}{k}.
    \label{eq:2}
  \end{align}
\end{prop}

In the context of TCP the above result lends itself to the
interpretation that even in the nonlinear, discrete-valued model there
is a stationary distribution of transmission rates---provided $k$ is
sufficiently large---so that the right-hand-side of~\eqref{eq:2} is less than
one.

\begin{remark}
  Note that another discrete version of~\eqref{AIMD disc} would be to consider an equation of the
  form
  \begin{align*}
    x_i(t) = A_i(x_i(n)) + B_i (T(n)), \quad  i \in \{1,\dots,d\},
  \end{align*}
  where both $A_i$ and $B_i$ are integer maps, that is $A_i,B_i : \Z_+^d \to \Z_+^d$. However, in
  the linear case, since the only additive maps on $\Z_+$ are of the form $x \mapsto m\cdot x$ with $m\in \Z_+$,
  there is no non-trivial additive map such that $A_i(x_i(n)) \le x_i(n)$. Thus, in such case there is no real
  analogue to a map of the form $A_i(x_i(n)) = \alpha_i x_i(n)$, $\alpha_i \in [0,1)$. \qede
\end{remark}

\begin{remark}
  Note that~\eqref{AIMD cont nonlin} is not the only way to generalize~\eqref{AIMD cont lin} to a
  nonlinear setting, cf., e.g.,~\cite{CKSW16}*{Ch.~11}. \qede
\end{remark}

\subsection{Wireless Communication Systems}

Consider a wireless, multi-user communication system in which transmitter power is allocated to provide each user with an
acceptable connection. Several such allocation models have been studied, see,
e.g.,~\cites{hanly1996-capacity-and-power-control-in-spread-spectrum-macrodiversity-radio-networks, Yat95}. Assume
that there are $d$ users and let $x=(x_1,\dots,x_d)$ denote the vector of transmitter power of the $d$
users. Also, let $I(x) = (I_1(x), \dots,I_d(x))$ denote the interference map, where $I_i(x)$ denotes the interference of other users that
user $i$ has to overcome. It is common to require that
\begin{align}\label{cond feas} x \ge I(x).
\end{align}
That is, every user has to employ transmission power which is at least as large as the interference. A vector
$x\in \R_+^d$ is said to be a feasible vector if it satisfies~\eqref{cond feas}, and a map $I$ is said to be
feasible if~\eqref{cond feas} has a feasible solution. Given the vector inequality~\eqref{cond feas}, one can consider also the iteration system
\begin{align}\label{control seq} x(n+1) = I(x(n)), \quad  n \in \Z_+.
\end{align}
Note that any fixed point of the system~\eqref{control seq} also satisfies the condition~\eqref{cond feas}.

Condition~\eqref{cond feas} arises from the so called \emph{Signal to Interference Ratio} (SIR), which can be described as follows. Assume that we are given $d$ users and $M$ base stations. As before, $x_j$ denotes the transmitted power of user $j$. Let $h_{k,j}$ denote the gain of user $j$ to base $k$. The received power signal from user $j$ at base $k$ is $h_{k,j}x_j$, and the interference seen by user $j$ at base $k$ is given by $\sum_{i\neq j}h_{k,i}x_i+\sigma_k$,
where $\sigma_k$ denotes the receiver noise at base $k$. Then, given a power vector $x = (x_1,\dots,x_d)$, the SIR of user $j$ at base station $k$, is given by
\begin{align}\label{def SIR}
\mu_{k,j}(x) = \frac{h_{k,j}}{\sum_{i\neq j}h_{k,i}x_i + \sigma_k}.
\end{align}
Since here we are interested in the study of integer maps, we will assume that $x_j, h_{k,j}, \sigma_k\in \N$ for all $j\in \{1,\dots,d\}$ and $k\in \{1,\dots,M\}$. In such case the SIR defined in~\eqref{def SIR} satisfies $\mu_{k,j}(x) \in \mathbb Q\cap [0,\infty)$.

One example of an interference function is the so called Fixed Assignment Interference, which can be described as follows. Assume that $a_j$ is the base assigned to user $j$. For $j\in \{1,\dots,d\}$, define
\begin{align*}
I_j(x) = \frac{\gamma_j}{\mu_{a_j,j}(x)} \stackrel{\eqref{def SIR}}{=} \gamma_j \frac{\sum_{i\neq j}h_{a_j,i}x_i+\sigma_{a_j}}{h_{a_j,j}},
\end{align*}
where $\gamma_j \in \R_+$. This case was considered for example in~\cites{GVGZ93, NA83}. If we assume as before that $x_i, h_{k,i} \in \N$ for all $i \in \{1,\dots,d\}$ and $k \in \{1,\dots,M\}$, then for all $j \in \{1,\dots,d\}$, we can choose $\gamma_j\in \Z_+$ such that $I_j(x) \in \Z_+$. In such case, $I$ is in fact an additively affine map, as defined in Section~\ref{sec concave}. Hence, there exists $b\in (0,\infty)$ such that $\|I(x)\|_1 \le b\|x\|_1$ for all $x\in \Z_+^d$. This follows for example from Proposition~\ref{prop bounded}. Therefore, by Theorem~\ref{thm concave case} and Corollary~\ref{cor sequence}, the following holds.

\begin{prop}
  Let $k \in \N$ which satisfies $k \ge d$, and assume that there
  exists $c\in (0,1/d]$ such that $I(x) \ge c\|I(x)\|_1\e$ for all
  $x\in \Z_+^d$ with $\|x\|_1=k$. Then there
  exists $\bar x \in \Z_+^d$ with $\|\bar x\|_1 = k$ such that
  \begin{align*}
    \left\|\frac{I(\bar x)}{\|I(\bar x)\|_1}-\frac{\bar x}{\|\bar x\|_1}\right\|_2 \le \frac{4dc^{-2}+2\sqrt d}{k}.
  \end{align*}
  Also, if $b\in (0,\infty)$ is such that $\|I(x)\|_1 \le b\|x\|_1$ for
  all $x\in \Z_+^d$ with $\|x\|_1=k$, then
  \begin{align}\label{almost feas}
    I(\bar x) \le b\left(1+\frac{4dc^{-2}+2\sqrt d}{ck}\right)\bar x.
  \end{align}
\end{prop}

\begin{remark}
Note that~\eqref{almost feas} is close to the feasibility condition~\eqref{cond feas}, especially when $b$ is not much larger than 1.\qede
\end{remark}

One can also consider a more general interference map. The following definition appeared in~\cite{Yat95}.

\begin{defin}\label{def standard}
  A map $I:\R_+^d \to \R_+^d$ is said to be \emph{standard} if the following conditions hold.
  \begin{itemize}
  \item Positivity: $I(x) \gg 0$ for all $x\in \R_+^d$. 
  \item Monotonicity: $I(x) \ge I(y)$ whenever $x \ge y$.
  \item Scalability: $I(\alpha x) \ll \alpha I(x)$ for all $x\in \R_+^d$ and $\alpha \in (1,\infty)$.
  \end{itemize}
\end{defin}
Recall that $x \ll y \iff x_i < y_i$ for all $i \in \{1,\dots,d\}$. The scalability property means
that if users have an acceptable connection under the vector $x$, then users will have a more than acceptable
connection if all powers are scaled up uniformly.

It was shown in~\cite{Yat95}*{Thm.~1,~Thm.~2} that if $I$ is standard and feasible, then~\eqref{control seq}
has a unique solution.

A discrete analogue of the scalability condition would be $m I(x) \ge I(mx)$ for all $m \in \N$ and $x\in
\Z_+^d$. The following proposition is immediate.

\begin{prop}
  Assume that $A:\Z_+^d \to \Z_+^d$ is concave. Then for all $x \in \Z_+^d$ and $m \in \N$,
  \begin{align*}
    A(mx) \le m Ax.
  \end{align*}
\end{prop}

\begin{proof}
  Write $mx = 1\cdot (mx) + (m-1)0$. Therefore, by the concavity property of $A$,
  \begin{align*}
    mAx \ge A(mx)+(m-1)A(0) \ge A(mx),
  \end{align*}
  and this completes the proof.
\end{proof}

In particular, it follows that every concave map on $\Z_+^d$ is scalable (even though we might not have a strict inequality as in Definition~\ref{def standard}). In such case we can also apply
Theorem~\ref{thm concave case} to obtain an approximate solution to~\eqref{control seq} for a more general interference map $I$.

\section{Conclusion \& Open questions}\label{sec conclusion}
This paper extends results from the Perron\textendash Frobenius theory to a
discrete setting and discusses some of the applications of such
extensions. We believe that further progress can be made in this
direction. In particular, the following questions remain open.

We do not know whether for some classes of maps one can obtain a
stronger quantitative bound in Theorem~\ref{main thm}.

As noted in Remark~\ref{rem choice}, the choice of the $\ell_1$ norm
is essential in the proof of Theorem~\ref{main thm}. This is in
contrast to the case of maps on $\R_+^d$~\cites{Koh82, Kra86}, where
any norm can be used. It would be interesting to know whether one can
obtain approximate eigenvectors without using a specific norm.

Many of the results of the Perron\textendash Frobenius theory for maps on
$\R_+^d$ remain true if the more general case of maps that leave a
cone invariant is considered, see e.g.~\cite{LN12}. We believe similar
generalizations can be obtained for integer maps.

It would be interesting to know whether a result in the spirit of
Theorem~\ref{main thm} holds for maps defined on other spaces, such as
infinite dimensional lattices or other commutative groups. Note that
in an infinite dimensional space, we do not have an equivalence
between the $\ell_2$ and the $\ell_\infty$ norms, which is crucial in
the proof of Theorem~\ref{main thm}.

Of practical relevance is the development of computational methods for
the efficient computation of approximate eigenvectors in the absence of
the power-method, c.f.~Remark~\ref{rem:power-method}.

\end{document}